\documentclass[11pt]{amsart}
\usepackage{graphicx}
\usepackage{amssymb}
\usepackage{amsmath}
\usepackage{amsthm}
\usepackage{mathtools}
\usepackage{tikz}
\usepackage{psfrag}
\usepackage{xcolor}

\newtheorem{thm}{Theorem}[section]

\newtheorem{coro}[thm]{Corollary}

\theoremstyle{definition}

\newcommand{\N}{\mathbb N}
\newcommand{\M}{\mathbb M}

\newcommand{\R}{\mathbb R}
\newcommand{\dist}{\operatorname{dist}}

\newcommand{\rank}{\operatorname{rank}}

\makeatletter
\numberwithin{equation}{section}
\makeatother

\begin{document}



       \author{Seonghak Kim}
       \address{Institute for Mathematical Sciences\\ Renmin University of China \\  Beijing 100872, PRC}
       \email{kimseo14@ruc.edu.cn}


       \author{Youngwoo Koh}

       \address{School of Mathematics\\ Korea Institute for Advanced Study \\ Seoul 130-722, ROK}


       \email{ywkoh@kias.re.kr}


       \title[1-D non-convex elastodynamics]{Weak solutions for one-dimensional  non-convex elastodynamics}

\subjclass[2010]{74B20,74N15,35M13}
\keywords{non-convex elastodynamics, hyperbolic-elliptic equations, phase transition, partial differential inclusions, Baire's category method, microstructures of weak solutions}

\begin{abstract}
We explore local existence and properties of classical weak solutions to the initial-boundary value problem of a one-dimensional quasilinear equation of elastodynamics with non-convex stored-energy function, a model of phase transitions in elastic bars  proposed by Ericksen \cite{Er}. The instantaneous formation of microstructures of local weak solutions is observed for all smooth initial data with  initial strain having its range overlapping with the phase transition zone of the Piola-Kirchhoff stress. As byproducts, it is shown that such a problem admits a local weak solution for all smooth initial data and local weak solutions that are smooth for a short period of time and exhibit microstructures thereafter for some smooth initial data. In a parallel exposition, we also include some results concerning one-dimensional quasilinear  hyperbolic-elliptic equations.
\end{abstract}
\maketitle

\section{Introduction}
The evolution process of a one-dimensional continuous medium with elastic response can be modeled by quasilinear wave equations of the form
\begin{equation}\label{main-P}
u_{tt} =\sigma(u_x)_x, 
\end{equation}
where $u=u(x,t)$ denotes the displacement of a reference point  $x$ at time $t$ and $\sigma=\sigma(s)$  the Piola-Kirchhoff stress, which is the derivative of a stored-energy function $W=W(s)\ge 0$. With $v=u_x$ and $w=u_t$, one may study equation (\ref{main-P}) as the system of conservation laws:
\begin{equation}\label{main-cons}
\left\{\begin{split}
v_t & = w_x, \\
w_t & = \sigma(v)_x.
\end{split}\right.
\end{equation}

In case of a strictly convex stored-energy function, the existence of weak or classical solutions to  equation (\ref{main-P}) and  its vectorial case has been studied extensively: Global weak solutions to system (\ref{main-cons}) and hence  equation (\ref{main-P}) are established in a classical work of {DiPerna} \cite{Di} via vanishing viscosity method  in the framework of compensated compactness of {Tartar} \cite{Ta} for $L^\infty$ data and later by {Lin} \cite{Li} and {Shearer} \cite{Sr} in an $L^p$ setup. This framework is also used to construct global weak solutions to (\ref{main-P}) via relaxation methods by {Serre} \cite{Se} and {Tzavaras} \cite{Tz}. An alternative variational scheme is studied by {Demoulini \emph{et al.}} \cite{DST} via time discretization. However global existence of weak solutions to the vectorial case of (\ref{main-P}) is still open. In regard to classical solutions to (\ref{main-P}) and its vectorial case, one can refer to {Dafermos and Hrusa} \cite{DH} for local existence of smooth solutions, to {Klainerman and Sideris} \cite{KS} for global existence of smooth solutions  for small initial data in dimension 3, and to {Dafermos} \cite{Ds1} for uniqueness of smooth solutions in the class of BV weak solutions whose shock intensity is not too strong.

Convexity assumption on the stored-energy function is often regarded as  a severe restriction in view of the actual behavior of elastic materials (see, e.g., \cite[Section 2]{Hi} and \cite[Section 8]{CN}). However there have not been  many analytic works dealing with the lack of convexity on the energy function. For the vectorial case of equation (\ref{main-P}) in dimension 3, measure-valued solutions are constructed for  polyconvex energy functions by {Demoulini \emph{et al.}} \cite{DST1}. Also by the same authors \cite{DST2}, in the identical situation, it is shown  that a dissipative measure-valued solution coincides with a strong one provided the latter exists. Assuming convexity on the energy function at infinity but not allowing polyconvexity, measure-valued solutions are obtained by {Rieger} \cite{Ri} for the vectorial case of (\ref{main-P}) in any dimension. Despite of all these existence results,   there has been no example of a non-convex energy function with which (\ref{main-P}) admits \emph{classical} weak solutions  in general, not to mention its vectorial case. Among some optimistic and pessimistic opinions, {Rieger} \cite{Ri} expects such solutions even showing microstructures of phase transitions. Moreover, such expected phenomenology is successfully implemented  in some numerical simulations \cite{CR, Pr}.

In this paper, we study weak solutions to the one-dimensional initial-boundary value problem of non-convex elastodynamics
\begin{equation}\label{ib-P}
\begin{cases} u_{tt} =\sigma(u_x)_x& \mbox{in $\Omega_T=\Omega\times (0,T)$,} \\
u(0,t)=u(1,t)=0 & \mbox{for $t\in(0,T)$,}\\
u =g,\,u_t=h & \mbox{on $\Omega\times \{t=0\}$},
\end{cases}
\end{equation}
where $\Omega=(0,1)\subset\R$ is the domain occupied by a reference configuration of an elastic bar, $T>0$ is a fixed number,   $g$ is the initial displacement of the bar, $h$ is the initial rate of change of the displacement, and the stress $\sigma:(-1,\infty)\to\R$ is given as in Figure \ref{fig1}. The zero boundary condition here amounts to the physical situation of fixing the end-points of the bar.  In this case, the energy function $W:(-1,\infty)\to[0,\infty)$ may satisfy $W(s)\to\infty$ as $s\to -1^+$; but this is not required to obtain our result.
On the other hand, we consider  (\ref{ib-P}) as a prototype of the hyperbolic-elliptic problem with a non-monotone constitutive function $\sigma:\R\to\R$ as in Figure \ref{fig2}.

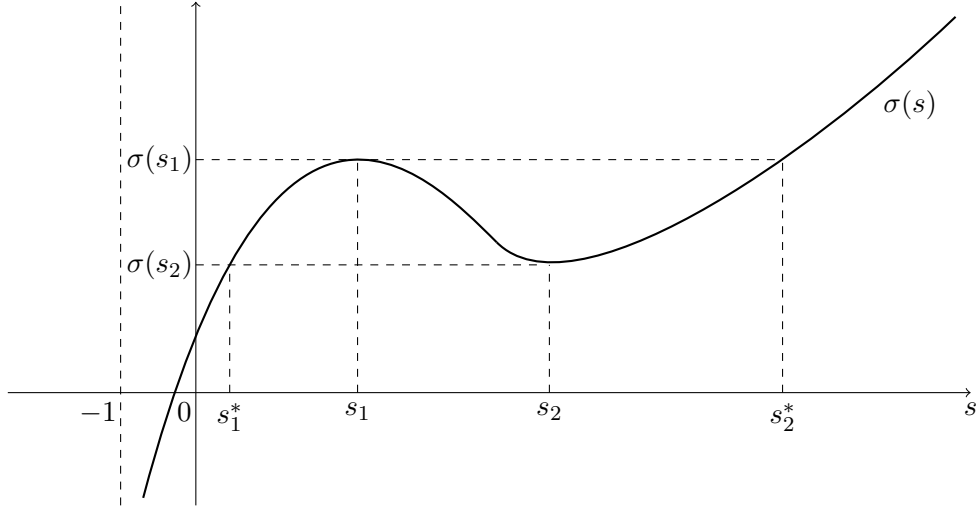
\begin{figure}[ht]
\begin{center}
\begin{tikzpicture}[scale =1]
    \draw[->] (-1.5,0) -- (11.3,0);
	\draw[dashed] (0,-1.5) -- (0,5.2);
    \draw[->] (1,-1.5) -- (1,5.2);
 \draw[dashed] (1,1.7)--(5.6,1.7);
 \draw[dashed] (1,3.1)--(8.8,3.1);
 \draw[dashed] (8.8,0)--(8.8,3.1);
    \draw[dashed] (5.7, 0)  --  (5.7, 1.7) ;
   \draw[dashed] (1.45, 0)  --  (1.45, 1.7) ;
   \draw[dashed] (3.15, 0)  --  (3.15, 3.15) ;
	\draw[thick]   (0.3, -1.4) .. controls (2,5) and  (4, 3)   ..(5,2);
	\draw[thick]   (5, 2) .. controls  (6, 1) and (9,3) ..(11.1, 5 );
	\draw (11.3,0) node[below] {$s$};
    \draw (-0.3,0) node[below] {{$-1$}};
    \draw (10.5, 3.5) node[above] {$\sigma(s)$};
 \draw (1.1, 3.1) node[left] {$\sigma(s_1)$};
 \draw (1.1, 1.7) node[left] {$\sigma(s_2)$};
   \draw (5.7, 0) node[below] {$s_2$};
   \draw (3.15, 0) node[below] {$s_1$};
   \draw (0.85, 0) node[below] {$0$};
   \draw (1.45, 0) node[below] {$s_1^*$};
   \draw (8.8, 0) node[below] {$s_2^*$};
    \end{tikzpicture}
\end{center}
\caption{Non-monotone Piola-Kirchhoff stress  $\sigma(s).$}
\label{fig1}
\end{figure}

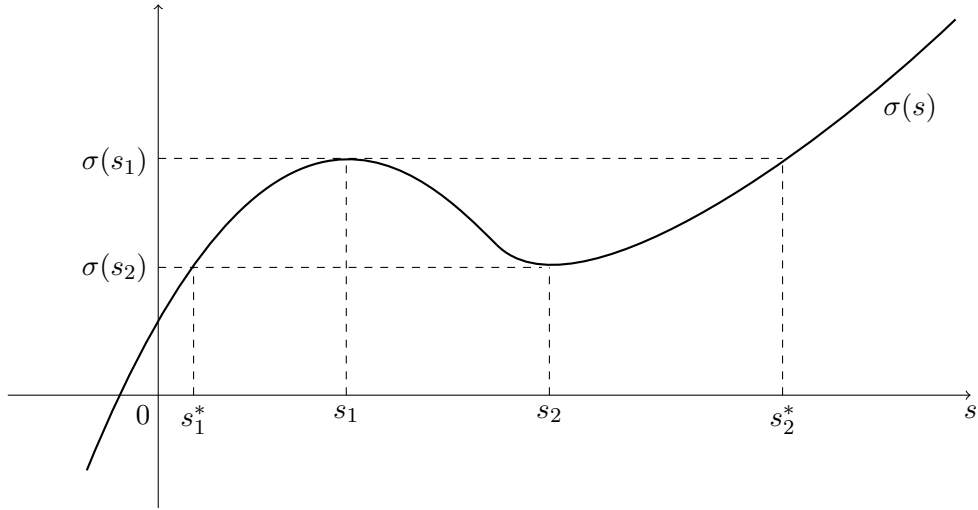
\begin{figure}[ht]
\begin{center}
\begin{tikzpicture}[scale =1]
    \draw[->] (-1.5,0) -- (11.3,0);
	\draw[->] (0.5,-1.5) -- (0.5,5.2);
 \draw[dashed] (0.5,1.7)--(5.6,1.7);
 \draw[dashed] (0.5,3.15)--(8.8,3.15);
 \draw[dashed] (8.8,0)--(8.8,3.1);
    \draw[dashed] (5.7, 0)  --  (5.7, 1.7) ;
   \draw[dashed] (0.97, 0)  --  (0.97, 1.7) ;
   \draw[dashed] (3, 0)  --  (3, 3.1) ;
	\draw[thick]   (-0.45, -1) .. controls (2,5) and  (4, 3)   ..(5,2);
	\draw[thick]   (5, 2) .. controls  (6, 1) and (9,3) ..(11.1, 5 );
	\draw (11.3,0) node[below] {$s$};
    \draw (0.3,0) node[below] {{$0$}};
    \draw (10.5, 3.5) node[above] {$\sigma(s)$};
 \draw (0.5, 3.1) node[left] {$\sigma(s_1)$};
 \draw (0.5, 1.7) node[left] {$\sigma(s_2)$};
   \draw (5.7, 0) node[below] {$s_2$};
   \draw (3, 0) node[below] {$s_1$};
   \draw (0.97, 0) node[below] {$s_1^*$};
   \draw (8.8, 0) node[below] {$s_2^*$};
    \end{tikzpicture}
\end{center}
\caption{Non-monotone constitutive function $\sigma(s).$}
\label{fig2}
\end{figure}

Problem (\ref{ib-P}) with a non-monotone stress $\sigma(s)$ as in Figure \ref{fig1} is proposed by {Ericksen} \cite{Er} as a model of the phenomena of phase transitions in elastic bars. There have been many studies on such a problem that usually fall into two types. One direction of study is to consider the Riemann problem of the system of conservation laws of mixed type (\ref{main-cons})    initiated by James \cite{Ja} and followed by numerous works (see, e.g., {Shearer} \cite{Sh}, {Pego and Serre} \cite{PS} and {Hattori} \cite{Ha}). Another path is to study the viscoelastic version of equation (\ref{main-P}): To name a few among initiative works, {Dafermos} \cite{Ds} considers the equation $u_{tt}=\sigma(u_x,u_{xt})_x+f(x,t)$ under certain parabolicity and growth conditions and establishes global existence and uniqueness of smooth solutions together with some asymptotic behaviors as $t\to\infty$. Following the work of {Andrews} \cite{An}, {Andrews and Ball} \cite{AB} prove global existence of weak solutions to the equation $u_{tt}=u_{xxt}+\sigma(u_x)_x$ for non-smooth initial data and study their long-time behaviors. For the same equation, {Pego} \cite{Pe} characterizes  long-time convergence of weak solutions to several different types of stationary states in a strong sense. Nonetheless, up to our best knowledge, the main theorem below may be the  first general existence result on weak solutions to  (\ref{ib-P}), not in the stream of the Riemann problem nor in that of non-convex viscoelastodynamics.


Let $\sigma(s)$ be given as in Figure \ref{fig1} or \ref{fig2} (see section \ref{sec:state}). For an initial datum $(g,h)\in W^{1,\infty}_0(\Omega)\times L^\infty(\Omega)$, we say that a function $u\in W^{1,\infty}(\Omega_T)$ is  a \emph{weak solution} to problem (\ref{ib-P}) provided $u_x>-1$ a.e. in $\Omega_T$  in case of Figure \ref{fig1}, for all $\varphi\in C^\infty_c(\Omega\times[0,T))$, we have
\begin{equation}\label{def:sol}
\int_{\Omega_T}(u_t\varphi_t-\sigma(u_x)\varphi_x)\,dxdt=-\int_0^1 h(x)\varphi(x,0)\,dx,
\end{equation}
and
\begin{equation}\label{def:sol-1}
\left\{\begin{array}{ll}
         u(0,t)=u(1,t)=0 & \mbox{for $t\in(0,T)$}, \\
         u(x,0)=g(x) & \mbox{for $x\in\Omega$}.
       \end{array}
\right.
\end{equation}

The main result of the paper is the following theorem that will be separated into two detailed statements in section \ref{sec:state} along with some corollaries.

\begin{thm}\label{thm:main-pre}
Let $\sigma(s)$ be as in Figure \ref{fig1} or \ref{fig2}, and let $(g,h)\in W^{3,2}_0(\Omega)\times W^{2,2}_0(\Omega) $ with $s_1^*<g' (x_0)<s_2^*$ at some $x_0\in\Omega$. In case of Figure \ref{fig1}, assume also $g'(x)>-1$ for all $x\in\bar\Omega$.  Then there exists a finite number $T>0$ for which problem (\ref{ib-P}) admits infinitely many weak solutions.
\end{thm}


Existence and non-uniqueness of weak solutions to problem (\ref{ib-P}) have been generally accepted (especially, in the context of the Riemann problem) and actively studied in the community of solid mechanics. Such non-uniqueness is usually understood to be arising from a constitutive deficiency in the theory of elastodynamics, reflecting the need to incorporate some additional relations (see, e.g., {Slemrod} \cite{Sl}, {Abeyaratne and Knowles} \cite{AK} and {Truskinovsky and Zanzotto} \cite{TZ}).

Global existence of Lipschitz continuous weak solutions to problem (\ref{ib-P}) is not directly obtained in the course of proving Theorem \ref{thm:main-pre} as it would require a global classical solution to some modified hyperbolic problem in our method of proof and such a global one might not exist due to a possible shock formation at a finite time. However, we still expect the existence of global $W^{1,p}$-solutions $(p<\infty)$ to (\ref{ib-P}).

The rest of the paper is organized as follows. Section \ref{sec:state} describes precise structural assumptions on the functions $\sigma(s)$ corresponding to Figures \ref{fig1} and \ref{fig2}, respectively. Then detailed statements of the main result, Theorem \ref{thm:main-pre}, with respect to Figures \ref{fig1} and \ref{fig2} are introduced separately as Theorems \ref{thm:main-NCE} and \ref{thm:main-HEP} with relevant corollaries in each case.
Section \ref{sec:exist} begins with a motivational approach to solve problem (\ref{ib-P}) as a homogeneous partial differential inclusion with a linear constraint. Then the main results in precise form, Theorems \ref{thm:main-NCE} and \ref{thm:main-HEP}, are proved at the same time under a pivotal density fact, Theorem \ref{thm:density}. The proofs of the corollaries to the main results are also included in section \ref{sec:exist}.
In section \ref{sec:rank-1}, a major tool for proving the density fact is established in a general form.
Lastly, section \ref{sec:density-proof} carries out the proof of the density fact.

In closing this section, we introduce some notations. Let $m,n$ be positive integers. We denote by $\M^{m\times n}$  the space of $m\times n$ real matrices and by $\M_{sym}^{n\times n}$ that of symmetric $n\times n$ real matrices. We use $O(n)$ to denote the space of $n\times n$ orthogonal matrices. For a given matrix $M\in\M^{m\times n}$, we write $M_{ij}$ for  the component of $M$ in the $i$th row and $j$th column and $M^T$ for the transpose of $M$. For a bounded domain $U\subset\R^n$ and a function $w^*\in W^{m,p}(U)$ $(1\le p\le\infty)$, we use  $W_{w^*}^{m,p}(U)$ to denote the space of functions $w\in W^{m,p}(U)$ with boundary trace $w^*.$


\section{Precise statements of main theorems}\label{sec:state}

In this section, we present structural assumptions on the functions $\sigma(s)$ for \textbf{Case I:} non-convex elastodynamics  and  \textbf{Case II:} hyperbolic-elliptic problem corresponding to Figures \ref{fig1} and \ref{fig2}, respectively. Then we give the detailed statement of the main result, Theorem \ref{thm:main-pre}, in each case, followed by some relevant byproducts.

\textbf{(Case I):} For the problem of non-convex elastodynamics, we impose the following conditions on the stress $\sigma:(-1,\infty)\to\R$ (see Figure \ref{fig1}).

\textbf{Hypothesis (NC):} There exist two numbers $s_2>s_1>-1$ with the following properties:
\begin{itemize}
\item[(a)] $\sigma\in C^3((-1,s_1)\cup(s_2,\infty))\cap C((-1,s_1]\cup[s_2,\infty))$.
\item[(b)] $\displaystyle\lim_{s\to -1^+}\sigma(s)=-\infty$.
\item[(c)] $\sigma:(s_1,s_2)\to\R$ is bounded and measurable.
\item[(d)] $\sigma(s_1)>\sigma(s_2)$, and $\sigma'(s)>0$ for all $s\in (-1,s_1)\cup(s_2,\infty)$.
\item[(e)] There exist two numbers $c>0$ and $s_1+1>\rho>0$  such that  $\sigma'(s)\ge c$ for all $s\in (-1,s_1-\rho]\cup[s_2+\rho,\infty)$.
\item[(f)]
Let $s_1^*\in (-1,s_1)$ and $s^*_2\in(s_2,\infty)$ denote the unique numbers with $\sigma(s_1^*)=\sigma(s_2)$ and $\sigma(s_2^*)=\sigma(s_1)$, respectively.
\end{itemize}

\textbf{(Case II):} For the hyperbolic-elliptic problem, we assume the following for the constitutive function $\sigma:\R\to\R$ (see Figure \ref{fig2}).

\textbf{Hypothesis (HE):} There exist two numbers $s_2>s_1$ satisfying the following:
\begin{itemize}
\item[(a)] $\sigma\in C^3((-\infty,s_1)\cup(s_2,\infty))\cap C((-\infty,s_1]\cup[s_2,\infty))$.
\item[(b)] $\sigma:(s_1,s_2)\to\R$ is bounded and measurable.
\item[(c)] $\sigma(s_1)>\sigma(s_2)$, and $\sigma'(s)>0$ for all $s\in (-\infty,s_1)\cup(s_2,\infty)$.
\item[(d)] There exists a number $c>0$  such that  $\sigma'(s)\ge c$ for all $s\in (-\infty,s_1-1]\cup[s_2+1,\infty)$.
\item[(e)]
Let $s_1^*\in (-\infty,s_1)$ and $s^*_2\in(s_2,\infty)$ denote the unique numbers with $\sigma(s_1^*)=\sigma(s_2)$ and $\sigma(s_2^*)=\sigma(s_1)$, respectively.
\end{itemize}

In both cases, for each $r\in (\sigma(s_2),\sigma(s_1))$, let $s_-(r)\in(s_1^*,s_1)$ and $s_+(r)\in(s_2,s_2^*)$ denote the unique numbers with $\sigma(s_\pm(r))=r.$ We may call the interval $(s_1^*,s_2^*)$ as the \emph{phase transition zone} of problem (\ref{ib-P}), since the formation of microstructures and breakdown of uniqueness of weak solutions to (\ref{ib-P}) begin to occur whenever the range of the  initial strain $g'$  overlaps with the interval $(s_1^*,s_2^*)$ as we can see below from Theorems \ref{thm:main-NCE} and \ref{thm:main-HEP} and their corollaries.

We now state the main result on \textbf{Case I:} weak solutions for non-convex elastodynamics under Hypothesis (NC).
\begin{thm}\label{thm:main-NCE}
Let $(g,h)\in W^{3,2}_0(\Omega)\times W^{2,2}_0(\Omega)$ satisfy $g'(x)>-1$ for all $x\in\bar\Omega$ and $s_1^*<g' (x_0)<s_2^*$ for some $x_0\in\Omega$. Let $\sigma(s_2)<r_1<r_2<\sigma(s_1)$ be any two numbers with $s_-(r_1)<g'(x_0)<s_+(r_2)$. Then there exist a finite number $T>0$, a function $\displaystyle{u^*\in \cap_{k=0}^3C^k([0,T];W^{3-k,2}_0(\Omega))}$ with $u^*_x>-1$ on $\bar\Omega_T$, where $W^{0,2}_0(\Omega)=L^2(\Omega)$, and three disjoint open sets $\Omega_T^1,\Omega_T^2,\Omega_T^3\subset \Omega_T$ with $\Omega_T^2\not=\emptyset$, $\partial\Omega_T^1\cap\partial\Omega_T^3=\emptyset$, and \begin{equation}\label{sep-domain-NCE}
\left\{\begin{array}{l}
         \partial\Omega_T^1\cap \Omega_0=\{(x,0)\,|\, x\in\Omega,\,g'(x)<s_-(r_1)\}, \\
         \partial\Omega_T^2\cap \Omega_0  =\{(x,0)\,|\, x\in\Omega,\,s_-(r_1)<g'(x)<s_+(r_2)\}, \\
         \partial\Omega_T^3\cap \Omega_0=\{(x,0)\,|\, x\in\Omega,\,g'(x)>s_+(r_2)\},
       \end{array}
 \right.
\end{equation}
where $\Omega_0=\Omega\times\{t=0\},$
such that for each $\epsilon>0$, there exist a number $T_\epsilon\in(0,T)$ and infinitely many weak solutions $u\in W^{1,\infty}_{u^*}(\Omega_T)$ to problem (\ref{ib-P}) with the following properties:
\begin{itemize}
\item[(a)] Approximate initial rate of change:
\[
\|u_t-h\|_{L^\infty(\Omega_{T_\epsilon})}<\epsilon,
\]
where $\Omega_{T_\epsilon}=\Omega\times(0,T_\epsilon)$.


\item[(b)] Classical part of $u$:
\begin{itemize}
\item[(i)] $u\equiv u^*$ on $\overline{\Omega_T^1\cup\Omega_T^3}$,

\item[(ii)] $u_t(x,0)=h(x)\quad\forall (x,0)\in(\partial\Omega_T^1\cup\partial\Omega_T^3)\cap\Omega_0$,

\item[(iii)] $u_x(x,t)\left\{\begin{array}{ll}
                        \in(-1,s_-(r_1)) & \forall(x,t)\in\Omega_T^1, \\
                        >s_+(r_2) & \forall(x,t)\in\Omega_T^3.
                      \end{array}
 \right.$
\end{itemize}

\item[(c)] Microstructure of $u$:


$u_x(x,t)\in[s_-(r_1),s_-(r_2)]\cup[s_+(r_1),s_+(r_2)]$, a.e. $(x,t)\in\Omega_T^2$.

\item[(d)] Interface of $u$:

$u_x(x,t)\in\{s_-(r_1),s_+(r_2)\}$, a.e. $(x,t)\in\Omega_T\setminus (\cup_{i=1}^3\Omega_T^i)$.

\end{itemize}
\end{thm}

As a remark, note that corresponding deformations of the elastic bar, $d(x,t)=u(x,t)+x$, should satisfy \[
d_x(x,t)=u_x(x,t)+1 >-1+1=0,\;\;\mbox{a.e. $(x,t)\in\Omega_T$;}
\]
this guarantees that for a.e. $t\in(0,T)$, such deformations $d:[0,1]\times\{t\}\to[0,1]$ are strictly increasing with $d(0,t)=0$ and $d(1,t)=1$. Moreover, for such a $t\in(0,T)$, the deformations $d(x,t)$ are smooth (as much as the initial displacement $g$) for the values of $x\in[0,1]$ for which slope $d_x(x,t)\in(0,s_-(r_1)+1)\cup(s_+(r_2)+1,\infty)$ and are Lipschitz elsewhere with $d_x(x,t)\in[s_-(r_1)+1,s_-(r_2)+1]\cup[s_+(r_1)+1,s_+(r_2)+1]$ a.e. Therefore, these dynamic deformations fulfill a natural physical requirement of invertibility   for the motion of an elastic bar not allowing interpenetration.

As byproducts of Theorem \ref{thm:main-NCE}, we obtain the following two results for non-convex elastodynamics. The first is on local existence of weak solutions to problem (\ref{ib-P}) for all  smooth initial data. The second gives local weak solutions to (\ref{ib-P}) that are all identical and smooth for a short period of time and then evolve microstructures along with the breakdown of uniqueness for some smooth initial data.

\begin{coro}\label{coro:weak-NCE}
For any initial datum $(g,h)\in W^{3,2}_0(\Omega)\times W^{2,2}_0(\Omega)$ with $g'>-1$ on $\bar\Omega$, there exists a finite number $T>0$ for which problem (\ref{ib-P}) has a weak solution.
\end{coro}

\begin{coro}\label{coro:weak-NCE-micro}
Let $(g,h)\in W^{3,2}_0(\Omega)\times W^{2,2}_0(\Omega)$ satisfy $g'>-1$ on $\bar\Omega$. Assume $\max_{\bar\Omega}g'\in (s_1^*,s_1)$ or $\min_{\bar\Omega}g'\in (s_2,s_2^*)$. Then there exist finite numbers $T>T'>0$ such that problem (\ref{ib-P}) admits infinitely many weak solutions that are all equal to some $\displaystyle{u^*\in \cap_{k=0}^3C^k([0,T'];W^{3-k,2}_0(\Omega))}$ in $\Omega_{T'}$ and evolve microstructures from $t=T'$  as in Theorem \ref{thm:main-NCE}.
\end{coro}

The following is the main result on \textbf{Case II:}  hyperbolic-elliptic equations under Hypothesis (HE).

\begin{thm}\label{thm:main-HEP}
Let $(g,h)\in W^{3,2}_0(\Omega)\times W^{2,2}_0(\Omega)$ with $s_1^*<g' (x_0)<s_2^*$ for some $x_0\in\Omega$. Let $\sigma(s_2)<r_1<r_2<\sigma(s_1)$ be any two numbers with $s_-(r_1)<g'(x_0)<s_+(r_2)$. Then there exist a finite number $T>0$, a function $\displaystyle{u^*\in \cap_{k=0}^3C^k([0,T];W^{3-k,2}_0(\Omega))}$, and three disjoint open sets $\Omega_T^1,\Omega_T^2,\Omega_T^3\subset \Omega_T$ with $\Omega_T^2\not=\emptyset$, $\partial\Omega_T^1\cap\partial\Omega_T^3=\emptyset$, and \begin{equation}\label{sep-domain-HEP}
\left\{\begin{array}{l}
         \partial\Omega_T^1\cap \Omega_0=\{(x,0)\,|\, x\in\Omega,\,g'(x)<s_-(r_1)\}, \\
         \partial\Omega_T^2\cap \Omega_0  =\{(x,0)\,|\, x\in\Omega,\,s_-(r_1)<g'(x)<s_+(r_2)\}, \\
         \partial\Omega_T^3\cap \Omega_0=\{(x,0)\,|\, x\in\Omega,\,g'(x)>s_+(r_2)\}
       \end{array}
 \right.
\end{equation}
such that for each $\epsilon>0$, there exist a number $T_\epsilon\in(0,T)$ and infinitely many weak solutions $u\in W^{1,\infty}_{u^*}(\Omega_T)$ to problem (\ref{ib-P}) satisfying the following properties:
\begin{itemize}
\item[(a)] Approximate initial rate of change:
\[
\|u_t-h\|_{L^\infty(\Omega_{T_\epsilon})}<\epsilon.
\]


\item[(b)] Classical part of $u$:
\begin{itemize}
\item[(i)] $u\equiv u^*$ on $\overline{\Omega_T^1\cup\Omega_T^3}$,

\item[(ii)] $u_t(x,0)=h(x)\quad\forall (x,0)\in(\partial\Omega_T^1\cup\partial\Omega_T^3)\cap\Omega_0$,

\item[(iii)] $u_x(x,t)\left\{\begin{array}{cc}
                        <s_-(r_1) & \forall(x,t)\in\Omega_T^1, \\
                        >s_+(r_2) & \forall(x,t)\in\Omega_T^3.
                      \end{array}
 \right.$
\end{itemize}

\item[(c)] Microstructure of $u$:


$u_x(x,t)\in[s_-(r_1),s_-(r_2)]\cup[s_+(r_1),s_+(r_2)]$, a.e. $(x,t)\in\Omega_T^2$.

\item[(d)] Interface of $u$:

$u_x(x,t)\in\{s_-(r_1),s_+(r_2)\}$, a.e. $(x,t)\in\Omega_T\setminus (\cup_{i=1}^3\Omega_T^i)$.

\end{itemize}
\end{thm}

We also have the following results similar to Corollaries \ref{coro:weak-NCE} and \ref{coro:weak-NCE-micro}.

\begin{coro}\label{coro:weak-HEP}
For any initial datum $(g,h)\in W^{3,2}_0(\Omega)\times W^{2,2}_0(\Omega)$, there exists a finite number $T>0$ for which  problem (\ref{ib-P}) has a weak solution.
\end{coro}

\begin{coro}\label{coro:weak-HEP-micro}
Let $(g,h)\in W^{3,2}_0(\Omega)\times W^{2,2}_0(\Omega)$ satisfy $\max_{\bar\Omega}g'\in (s_1^*,s_1)$ or $\min_{\bar\Omega}g'\in (s_2,s_2^*)$. Then there exist finite numbers $T>T'>0$ such that problem (\ref{ib-P}) admits infinitely many weak solutions that are all equal to some $u^*\in \cap_{k=0}^3C^k([0,T'];$
$W^{3-k,2}_0(\Omega))$ in $\Omega_{T'}$ and evolve microstructures from $t=T'$  as in Theorem \ref{thm:main-HEP}.
\end{coro}

\section{Proof of main theorems}\label{sec:exist}
In this section, we prove the main results, Theorems \ref{thm:main-NCE} and \ref{thm:main-HEP}, with some essential ingredient, Theorem \ref{thm:density}, to be verified in sections \ref{sec:rank-1} and \ref{sec:density-proof}. The proofs of related corollaries are also included. 

Our exposition hereafter will be parallelwise for \textbf{Cases I} and \textbf{II}.

\subsection{An approach by differential inclusion}
We begin with a motivational approach to attack problem (\ref{ib-P}) for both \textbf{Cases I} and \textbf{II}. To solve  equation (\ref{main-P}) in the sense of distributions in $\Omega_T$, suppose there exists a vector function $w=(u,v)\in W^{1,\infty}(\Omega_T;\R^2)$ such that
\begin{equation}\label{abs-1}
v_x=u_t\quad\mbox{and}\quad v_t=\sigma(u_x)\quad\mbox{a.e. in $\Omega_T$}.
\end{equation}
We remark that this formulation is motivated by the approach in \cite{Zh} and  different from the usual setup of conservation laws (\ref{main-cons}).
For all $\varphi\in C^\infty_c(\Omega_T)$, we now have
\[
\int_{\Omega_T} u_t\varphi_t\,dxdt  =\int_{\Omega_T}v_x\varphi_t\,dxdt= \int_{\Omega_T}v_t\varphi_x\,dxdt= \int_{\Omega_T}\sigma(u_x)\varphi_x\,dxdt;
\]
hence having (\ref{abs-1}) is sufficient to solve (\ref{main-P}) in the sense of distributions in $\Omega_T$. Equivalently, we can rewrite (\ref{abs-1}) as
\[
\nabla w=\begin{pmatrix}
            u_x & u_t \\
            v_x & v_t
            \end{pmatrix}=
            \begin{pmatrix}
            u_x & v_x \\
            v_x & \sigma(u_x)
            \end{pmatrix}\quad\mbox{a.e. in $\Omega_T$},
\]
where $\nabla$ denotes the space-time gradient.
Set
\[
\Sigma_\sigma=\left\{\begin{pmatrix}
            s & b \\
            b & \sigma(s)
            \end{pmatrix}\in\M^{2\times 2}_{sym}\,\Big|\,s,b\in\R \right\}.
\]
We can now recast (\ref{abs-1}) as a  homogeneous  partial differential inclusion with a linear constraint:
\begin{equation*}
\nabla w(x,t)\in\Sigma_\sigma,\quad\mbox{a.e. $(x,t)\in\Omega_T$}.
\end{equation*}
We will solve this inclusion for a suitable subset $K$ of $\Sigma_\sigma$ to incorporate some detailed properties of weak solutions to (\ref{ib-P}).

Homogeneous   differential inclusions of the form $\nabla w\in K\subset\M^{m\times n}$ are first encountered and successfully understood  in the study of crystal microstructure  by {Ball and James} \cite{BJ}, {Chipot and Kinderlehrer} \cite{CK} and with a constraint on a minor of $\nabla w$ by {M\"uller and \v Sver\'ak}  \cite{MSv1}.
General inhomogeneous  differential inclusions are studied by {Dacorogna and Marcellini} \cite{DM1} using Baire's category method and by {M\"uller and  Sychev} \cite{MSy} using the method of convex integration; see also \cite{Ki}.     Moreover, the methods  of  differential inclusions  have been applied to other important problems concerning  elliptic systems \cite{MSv2}, Euler equations \cite{DS}, the porous media equation \cite{CFG}, the active scalar equation \cite{Sy}, Perona-Malik equation and its generalizations \cite{Zh,KY,KY1,KY2}, and ferromagnetism \cite{Ya1}.

\subsection{Proof of Theorems \ref{thm:main-NCE} and \ref{thm:main-HEP}}\label{subsec:mainproof}
Due to the similarity between  Theorems \ref{thm:main-NCE} and \ref{thm:main-HEP}, we can combine their proofs into a single one.

To start the proof, we assume functions $g,h$ and numbers $r_1,r_2$ are given as in Theorem \ref{thm:main-NCE} \textbf{(Case I)}, as in Theorem \ref{thm:main-HEP}  \textbf{(Case II)}.
For clarity, we divide the proof into several steps.

\textbf{(Modified hyperbolic problem):}
Using elementary calculus, from Hypothesis (NC) \textbf{(Case I)}, Hypothesis (HE) \textbf{(Case II)}, we can find a function $\sigma^*\in C^3(-1,\infty)$ \textbf{(Case I)}, $\sigma^*\in C^3(\R)$ \textbf{(Case II)} such that
\begin{equation}\label{modi}
\left\{\begin{array}{l}
         \mbox{$\sigma^*(s)=\sigma(s)$ for all $s\in(-1,s_-(r_1)]\cup[s_+(r_2),\infty)$ \textbf{(Case I)},} \\
         \mbox{\quad\quad\quad\quad\quad\,\,\, for all $s\in(-\infty,s_-(r_1)]\cup[s_+(r_2),\infty)$ \textbf{(Case II)},} \\
         \mbox{$(\sigma^*)'(s)\ge c^*$ for all $s\in(-1,\infty)$, for some constant $c^*>0$ \textbf{(Case I)},} \\
         \mbox{\quad\quad\quad\quad\quad\,\,\, for all $s\in\R$, for some constant $c^*>0$ \textbf{(Case II)},} \\
         \mbox{$\sigma^*(s)<\sigma(s)$ for all $s_-(r_1)<s\le s_-(r_2)$, and}\\
         \mbox{$\sigma^*(s)>\sigma(s)$ for all $s_+(r_1)\le s< s_+(r_2)$ (see Figure \ref{fig3} for both cases).}
       \end{array}
 \right.
\end{equation}

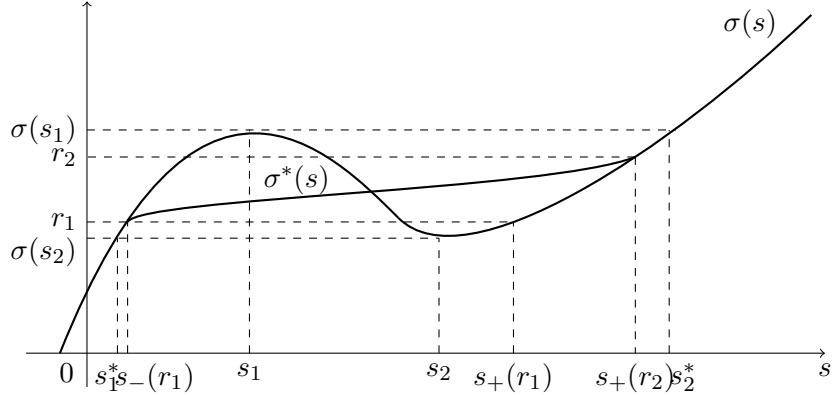
\begin{figure}[ht]
\begin{center}
\begin{tikzpicture}[scale = 0.9]
    \draw[->] (-.5,0) -- (11.3,0);
	\draw[->] (0.4,-.5) -- (0.4,5.2);
 \draw[dashed] (0.4,1.7)--(5.6,1.7);
 \draw[dashed] (0.4,3.3)--(9.1,3.3);
 \draw[dashed] (9,0)--(9,3.3);
    \draw[dashed] (5.6, 0)  --  (5.6, 1.7) ;
   \draw[dashed] (0.85, 0)  --  (0.85, 1.7) ;
\draw[dashed] (1, 0)  --  (1, 1.94) ;
\draw[dashed] (0.4,1.94)  --  (6.7, 1.94) ;
\draw[dashed] (6.7,0)  --  (6.7, 1.94) ;
   \draw[dashed] (2.8, 0)  --  (2.8, 3.3) ;
	\draw[thick]   (0, 0) .. controls (2,5) and  (4, 3)   ..(5,2);
          \draw[thick]   (1, 1.94) .. controls (1.2,2.3) and  (7.2, 2.4)   ..(8.5,2.9);
	\draw[thick]   (5, 2) .. controls  (6, 1) and (9,3) ..(11.1, 5 );
	\draw (11.3,0) node[below] {$s$};
    \draw (0.1,0) node[below] {{$0$}};
    \draw (10.2, 4.5) node[above] {$\sigma(s)$};
 \draw (0.4, 3.3) node[left] {$\sigma(s_1)$};
 \draw (0.4, 1.5) node[left] {$\sigma(s_2)$};
 \draw (0.4, 1.94) node[left] {$r_1$};
 \draw (0.4, 2.9) node[left] {$r_2$};
\draw[dashed] (8.5,2.9)--(0.4,2.9);
\draw[dashed] (8.5,2.9)--(8.5, 0);
   \draw (5.6, 0) node[below] {$s_2$};
   \draw (3.5, 2.2) node[above] {$\sigma^*(s)$};
   \draw (2.8, 0) node[below] {$s_1$};
   \draw (0.7, 0) node[below] {$s_1^*$};
   \draw (1.4, 0) node[below] {$s_-(r_1)$};
   \draw (9.2, 0) node[below] {$s_2^*$};
   \draw (8.5, 0) node[below] {$s_+(r_2)$};
   \draw (6.7, 0) node[below] {$s_+(r_1)$};
    \end{tikzpicture}
\end{center}
\caption{The original $\sigma(s)$ and modified $\sigma^*(s)$}
\label{fig3}
\end{figure}

Thanks to \cite[Theorem 5.2]{DH} \textbf{(Case I)}, \cite[Theorem 5.1]{DH} \textbf{(Case II)}, there exists a finite number $T>0$ such that the \emph{modified} initial-boundary value problem
\begin{equation}\label{ib-P-modi}
\begin{cases} u^*_{tt} =\sigma^*(u^*_x)_x& \mbox{in $\Omega_T$,} \\
u^*(0,t)=u^*(1,t)=0 & \mbox{for $t\in(0,T)$,}\\
u^* =g,\;u^*_t=h & \mbox{on $\Omega\times \{t=0\}$}
\end{cases}
\end{equation}
admits a unique solution $u^*\in \cap_{k=0}^3C^k([0,T];W^{3-k,2}_0(\Omega))$, with $u^*_x>-1$ on $\bar\Omega_T$ for \textbf{Case I}. By the Sobolev embedding theorem, we have $u^*\in C^2(\bar\Omega_T)$.  Let
\[\left\{
\begin{split}
\Omega^1_T & =\{(x,t)\in\Omega_T\,|\,u^*_x(x,t)<s_-(r_1)\},\\
\Omega^2_T & =\{(x,t)\in\Omega_T\,|\,s_-(r_1)<u^*_x(x,t)<s_+(r_2)\},\\
\Omega^3_T & =\{(x,t)\in\Omega_T\,|\,u^*_x(x,t)>s_+(r_2)\},\\
F_T & =\Omega_T\setminus(\cup_{i=1}^3\Omega_T^i);
\end{split}\right.
\]
then  (\ref{sep-domain-NCE}) holds \textbf{(Case I)}, (\ref{sep-domain-HEP}) holds \textbf{(Case II)}, and $\partial\Omega_T^1\cap\partial\Omega_T^3=\emptyset$.  As $s_-(r_1)<g'(x_0)=u^*_x(x_0,0)<s_+(r_2)$, we also have $\Omega_T^2\not=\emptyset$; so $|\Omega_T^2|>0$.

We define
\[
v^*(x,t)=\int_0^x h(z)\,dz+\int_0^t \sigma^*(u^*_x(x,\tau))\,d\tau\quad\forall(x,t)\in\Omega_T.
\]
Then $w^*:=(u^*,v^*)$ satisfies
\begin{equation}\label{classic}
v^*_x=u^*_t\quad\mbox{and}\quad v^*_t=\sigma^*(u^*_x)\quad\mbox{in $\Omega_T$}.
\end{equation}
Note that this   implies $v^*\in C^2(\bar\Omega_T)$; hence $w^*\in C^2(\bar\Omega_T;\R^2)$.

\textbf{(Related matrix sets):}
Define the sets (see Figure \ref{fig3})
\[
\begin{split}
\tilde K_\pm & =\{(s,\sigma(s))\in\R^2\,|\,s_\pm(r_1)\le s\le s_\pm(r_2)\},\\
\tilde K & = \tilde K_+\cup\tilde K_-,\\
\tilde U & =\{(s,r)\in\R^2\,|\, r_1<r<r_2,\,0<\lambda<1,\,s=\lambda s_-(r)+(1-\lambda) s_+(r) \},\\
K & = \left\{\begin{pmatrix} s & b \\ b & r \end{pmatrix}\in\M^{2\times 2}_{sym}\,\Big|\, (s,r)\in\tilde K,\, |b|\le\gamma\right\},\\
U & = \left\{\begin{pmatrix} s & b \\ b & r \end{pmatrix}\in\M^{2\times 2}_{sym}\,\Big|\, (s,r)\in\tilde U,\, |b|<\gamma \right\},
\end{split}
\]
where $\gamma:=\|u^*_t\|_{L^\infty(\Omega_T)}+1$.

\textbf{(Admissible class):}
Let $\epsilon>0$ be given. Choose a number $T_\epsilon\in(0,T]$ so that $\|u^*_t-h\|_{L^\infty(\Omega_{T_\epsilon})}<\epsilon/2=:\epsilon'$.
We then define the \emph{admissible class} $\mathcal A$ by
\[
\mathcal A=\left\{w=(u,v)\in W^{1,\infty}_{w^*}(\Omega_T;\R^2)\,\Bigg|\, \begin{array}{l}
                                                                          w\in C^2(\bar\Omega_T;\R^2),\,w\equiv w^* \\
                                                                          \mbox{in}\;\Omega_T\setminus\bar\Omega_T^w\;\mbox{for some open set}\\
                                                                          \mbox{$\Omega_T^w\subset\subset\Omega_T^2$ with $|\partial\Omega_T^w|=0,$}\\
                                                                          \nabla w(x,t) \in U\;\forall(x,t)\in\Omega_{T}^2,\\
                                                                          \|u_t-h\|_{L^\infty(\Omega_{T_\epsilon})}<\epsilon'
                                                                        \end{array}
 \right\}.
\]
It is easy to see from (\ref{modi}) and (\ref{classic}) that $w^*\in\mathcal A\not =\emptyset$.
For each $\delta>0$, we also define the \emph{$\delta$-approximating class} $\mathcal A_\delta$ by
\[
\mathcal A_\delta=\left\{w\in\mathcal A\,\Big|\, \int_{\Omega_T^2}\dist(\nabla w(x,t),K)\,dxdt\le\delta|\Omega_T^2|
 \right\}.
\]

\textbf{(Density result):}
One crucial step for the proof of Theorem \ref{thm:main-NCE} \textbf{(Case I)}, Theorem \ref{thm:main-HEP} \textbf{(Case II)} is the following density fact whose proof appears in section \ref{sec:density-proof} that is common for both cases.
\begin{thm}\label{thm:density}
For each $\delta>0$,
\[\mbox{$\mathcal A_\delta$ is dense in $\mathcal A$ with respect to the $L^\infty(\Omega_T;\R^2)$-norm.}\]
\end{thm}

\textbf{(Baire's category method):}
Let $\mathcal X$ denote the closure of $\mathcal A$ in the space $L^\infty(\Omega_T;\R^2)$, so that $(\mathcal X,L^\infty)$ is a nonempty complete metric space. As $U$ is bounded in $\M^{2\times 2}$, so is $\mathcal A$   in $W^{1,\infty}(\Omega_T;\R^2)$; thus it is easily checked that
\[
\mathcal X\subset W^{1,\infty}_{w^*}(\Omega_T;\R^2).
\]
Note that the space-time gradient operator $\nabla:\mathcal X\to L^1(\Omega_T;\M^{2\times 2})$ is a Baire-one function (see, e.g., \cite[Proposition 10.17]{Da}). So by the Baire Category Theorem (see, e.g., \cite[Theorem 10.15]{Da}), the set of points of discontinuity of the operator $\nabla$, say $\mathcal D_{\nabla}$, is a set of the first category; thus the set of points at which $\nabla$ is continuous, that is, $\mathcal C_{\nabla}:=\mathcal X\setminus\mathcal D_{\nabla}$, is dense in $\mathcal X$.

\textbf{(Completion of proof):}
Let us confirm that for any function $w=(u,v)\in\mathcal C_\nabla$, its first component $u$ is a weak solution to (\ref{ib-P}) satisfying (a)--(d). Towards this, fix any $w=(u,v)\in\mathcal C_\nabla$.

\textbf{\underline{(\ref{def:sol}) \& (\ref{def:sol-1}):}} To verify (\ref{def:sol}), let $\varphi\in C^\infty_c(\Omega\times[0,T))$. From  Theorem \ref{thm:density} and the density of $\mathcal A$ in $\mathcal X$, we can choose a sequence $w_k=(u_k,v_k)\in\mathcal A_{1/k}$ such that $w_k\to w$ in $\mathcal X$ as $k\to\infty$. As $w\in\mathcal C_{\nabla}$, we have $\nabla w_k\to \nabla w$ in $L^1(\Omega_T;\M^{2\times 2})$ and so pointwise a.e. in $\Omega_T$ after passing to a subsequence if necessary. By (\ref{classic}) and the definition of $\mathcal A$, we have $(v_k)_x=(u_k)_t$ in $\Omega_T$ and $(v_k)_x(x,0)=v^*_x(x,0)=u^*_t(x,0)=h(x)$ ($x\in\Omega$); so
\[
\begin{split}
\int_{\Omega_T}(u_k)_t\varphi_t\,dxdt & =\int_{\Omega_T}(v_k)_x\varphi_t\,dxdt\\
& = -\int_{\Omega_T}(v_k)_{xt}\varphi \,dxdt -\int_0^1 (v_k)_x(x,0)\varphi(x,0)\,dx\\
& = \int_{\Omega_T}(v_k)_{t}\varphi_x \,dxdt -\int_0^1 h(x)\varphi(x,0)\,dx,
\end{split}
\]
that is,
\[
\int_{\Omega_T}((u_k)_t\varphi_t-(v_k)_{t}\varphi_x)\,dxdt = -\int_0^1 h(x)\varphi(x,0)\,dx.
\]
On the other hand, by the Dominated Convergence Theorem, we have
\[
\int_{\Omega_T}((u_k)_t\varphi_t-(v_k)_{t}\varphi_x)\,dxdt \to \int_{\Omega_T}(u_t\varphi_t-v_{t}\varphi_x)\,dxdt;
\]
thus
\begin{equation}\label{pre-weak}
\int_{\Omega_T}(u_t\varphi_t-v_{t}\varphi_x)\,dxdt = -\int_0^1 h(x)\varphi(x,0)\,dx.
\end{equation}
Also, by the Dominated Convergence Theorem,
\[
\int_{\Omega_T^2} \dist(\nabla w_k(x,t),K)\,dxdt\to\int_{\Omega_T^2} \dist(\nabla w(x,t),K)\,dxdt.
\]
From the choice $w_k\in\mathcal A_{1/k}$, we have
\[
\int_{\Omega_T^2} \dist(\nabla w_k(x,t),K)\,dxdt\le\frac{|\Omega_T^2|}{k}\to 0;
\]
so
\[
\int_{\Omega_T^2} \dist(\nabla w(x,t),K)\,dxdt=0.
\]
Since $K$ is closed, we must have
\begin{equation}\label{inclusion}
\nabla w(x,t)\in K\subset \Sigma_\sigma,\quad\mbox{a.e. $(x,t)\in\Omega_T^2$}.
\end{equation}
For each $k$, we have $w_k\equiv w^*$ in $\Omega_T\setminus\bar\Omega_T^{w_k}$ for some open set $\Omega_T^{w_k}\subset\subset\Omega_T^2$ with $|\partial \Omega_T^{w_k}|=0$, and so  $\nabla w_k\equiv\nabla w^*$ in $\Omega_T\setminus\bar\Omega_T^{w_k}$; thus $w=w^*$ and $\nabla w=\nabla w^*$ a.e. in $\Omega_T\setminus\Omega_T^2$.
By (\ref{modi}) and (\ref{classic}), we have
\begin{equation*}
v_x=u_t\quad\mbox{and}\quad v_t=\sigma^*(u^*_x)=\sigma(u_x)\quad\mbox{a.e. in $\Omega_T\setminus\Omega_T^2$}.
\end{equation*}
This together with (\ref{inclusion}) implies that $\nabla w\in\Sigma_\sigma$ a.e. in $\Omega_T$. In particular, $v_t =\sigma(u_x)$ a.e. in $\Omega_T$. Reflecting this to (\ref{pre-weak}), we have (\ref{def:sol}). As $w=w^*$ on $\partial\Omega_T$, we also have (\ref{def:sol-1}).

\textbf{\underline{(a), (b), (c) \& (d):}}
As $w=w^*$ a.e. in $\Omega_T\setminus\Omega_T^2$, it follows from the continuity that $u\equiv u^*$ in $\Omega_T^1\cup\Omega_T^3$; so (b) is guaranteed by the definition of $\Omega_T^1$ and $\Omega_T^3$, with $u^*_x>-1$ on $\bar\Omega_T$ for \textbf{Case I}. Since $\nabla w=\nabla w^*$ a.e. in $\Omega_T\setminus\Omega_T^2$, we have $u_x=u^*_x\in\{s_-(r_1),s_+(r_2)\}$ a.e. in $F_T$; so (d) holds.
From $w_k\in\mathcal A_{1/k}\subset\mathcal A$, we have $|(u_k)_t(x,t)-h(x)|<\epsilon'$ for a.e. $(x,t)\in\Omega_{T_\epsilon}$. Taking the limit as $k\to\infty$, we obtain that $|u_t(x,t)-h(x)|\le\epsilon'$ for a.e. $(x,t)\in\Omega_{T_\epsilon}$; hence $\|u_t-h\|_{L^\infty(\Omega_{T_\epsilon})}\le\epsilon'=\epsilon/2<\epsilon$. Thus (a) is proved. From (\ref{inclusion}) and the definition of $K$, (c) follows.

\textbf{\underline{Infinitely many weak solutions:}}
Having shown that the first component $u$ of each pair $w=(u,v)$ in $\mathcal C_\nabla$ is a weak solution to (\ref{ib-P}) satisfying (a)--(d), it remains to verify that $\mathcal C_\nabla$ has infinitely many elements and that no two different pairs in $\mathcal C_\nabla$ have the first components that are equal.
Suppose on the contrary that $\mathcal C_\nabla$ has finitely many elements. Then $w^*\in\mathcal A\subset\mathcal X=\bar{\mathcal C}_\nabla=\mathcal C_\nabla$, and so $u^*$ itself is  a weak solution to (\ref{ib-P}) satisfying (a)--(d); this is a contradiction. Thus $\mathcal C_\nabla$ has infinitely many elements. Next, we check that for any two $w_1=(u_1,v_1),w_2=(u_2,v_2)\in\mathcal C_\nabla,$
\[
u_1=u_2\quad \Leftrightarrow\quad v_1=v_2.
\]
Suppose $u_1\equiv u_2$ in $\Omega_T$.
As $\nabla w_1,\nabla w_2\in\Sigma_\sigma$ a.e. in $\Omega_T$, we have, in particular, that
\[
(v_1)_x=(u_1)_t=(u_2)_t=(v_2)_x\;\;\mbox{a.e. in $\Omega_T$.}
\]
Since  both $v_1$ and $v_2$ share the same trace $v^*$ on $\partial\Omega_T$, it follows that $v_1\equiv v_2$ in $\Omega_T$. The converse can be shown  similarly. We can now conclude that there are infinitely many weak solutions to (\ref{ib-P}) satisfying (a)--(d).

The proof of Theorem \ref{thm:main-NCE} \textbf{(Case I)}, Theorem \ref{thm:main-HEP} \textbf{(Case II)} is now complete under  the density fact, Theorem \ref{thm:density}, to be justified in sections \ref{sec:rank-1} and \ref{sec:density-proof}.

\subsection{Proofs of Corollaries \ref{coro:weak-NCE}, \ref{coro:weak-NCE-micro}, \ref{coro:weak-HEP} and \ref{coro:weak-HEP-micro}}
We proceed the proofs of the companion versions of Corollaries \ref{coro:weak-NCE} and \ref{coro:weak-HEP} and Corollaries \ref{coro:weak-NCE-micro} and \ref{coro:weak-HEP-micro}, respectively.

\begin{proof}[Proof of Corollaries \ref{coro:weak-NCE} and \ref{coro:weak-HEP}]
Let $(g,h)\in W^{3,2}_0(\Omega)\times W^{2,2}_0(\Omega)$ be any given initial datum, with $g'>-1$ on $\bar\Omega$ for \textbf{Case I}. If $g'(x_0)\in(s_1^*,s_2^*)$ for some $x_0\in\Omega$, then the result follows immediately from Theorem \ref{thm:main-NCE} \textbf{(Case I)}, Theorem \ref{thm:main-HEP} \textbf{(Case II)}.

Next, let us assume $g'(x)\not\in(s_1^*,s_2^*)$ for all $x\in\bar\Omega.$ We may only consider the case that $g'(x)\ge s_2^*$ for all $x\in\bar\Omega$ as the other case can be shown similarly. Fix any two $\sigma(s_2)<r_1<r_2<\sigma(s_1),$ and choose a function $\sigma^*\in C^3(-1,\infty)$ \textbf{(Case I)}, $\sigma^*\in C^3(\R)$ \textbf{(Case II)} in such a way that (\ref{modi}) is fulfilled. By \cite[Theorem 5.2]{DH} \textbf{(Case I)}, \cite[Theorem 5.1]{DH} \textbf{(Case II)}, there exists a finite number $\tilde T>0$ such that the modified initial-boundary value problem (\ref{ib-P-modi}), with $T$ replaced by $\tilde T$, admits a unique solution $u^*\in \cap_{k=0}^3C^k([0,\tilde T];W^{3-k,2}_0(\Omega))$, with $u^*_x>-1$ on $\bar\Omega_{\tilde T}$ for \textbf{Case I}. Now, choose a number $0<T\le\tilde T$ so that $u^*_x\ge s_+(r_2)$ on $\bar\Omega_T$. Then $u^*$ itself is a classical and thus weak solution to problem (\ref{ib-P}).
\end{proof}

\begin{proof}[Proof of Corollaries \ref{coro:weak-NCE-micro} and \ref{coro:weak-HEP-micro}]
Let $(g,h)\in W^{3,2}_0(\Omega)\times W^{2,2}_0(\Omega)$ satisfy $\max_{\bar\Omega}g'\in (s_1^*,s_1)$ or $\min_{\bar\Omega}g'\in (s_2,s_2^*)$. In \textbf{Case I}, assume also $g'>-1$ on $\bar\Omega.$ We may only consider the case that $M:=\max_{\bar\Omega}g'\in (s_1^*,s_1)$ as the other case can be handled in a similar way. Choose two numbers $\sigma(s_2)<r_1<r_2<\sigma(s_1)$ so that $s_-(r_1)>M.$ Then take a $C^3$ function $\sigma^*(s)$ satisfying (\ref{modi}). Using \cite[Theorem 5.2]{DH} \textbf{(Case I)}, \cite[Theorem 5.1]{DH} \textbf{(Case II)}, we can find a finite number $\tilde T>0$ such that  modified  problem (\ref{ib-P-modi}), with $T$ replaced by $\tilde T$, has a unique solution $u^*\in \cap_{k=0}^3C^k([0,\tilde T];W^{3-k,2}_0(\Omega))$, with $u^*_x>-1$ on $\bar\Omega_{\tilde T}$ for \textbf{Case I}. Then choose a number $0<T'\le\tilde T$ so small that $u^*_x\le s_-(r_1)$ on $\bar\Omega_{T'}$ and that $s_1^*<u^*_x(x_0,T')$ for some $x_0\in\Omega$. With the initial datum $(u^*(\cdot,T'),u^*_t(\cdot,T'))\in W^{3,2}_0(\Omega)\times W^{2,2}_0(\Omega)$ at $t=T'$, with $u^*_x(\cdot,T')>-1$ on $\bar\Omega$ for \textbf{Case I}, we can apply Theorem \ref{thm:main-NCE} \textbf{(Case I)}, Theorem \ref{thm:main-HEP} \textbf{(Case II)} to obtain, for some finite number $T>T'$, infinitely many weak solutions $\tilde u\in W^{1,\infty}(\Omega\times(T',T))$ to the initial-boundary value problem
\begin{equation*}
\begin{cases} \tilde u_{tt} =\sigma(\tilde u_x)_x& \mbox{in $\Omega\times (T',T)$,} \\
\tilde u(0,t)=\tilde u(1,t)=0 & \mbox{for $t\in(T',T)$,}\\
\tilde u =u^*,\,\tilde u_t=u^*_t & \mbox{on $\Omega\times \{t=T'\}$}
\end{cases}
\end{equation*}
satisfying the stated properties in the theorem.
Then the glued functions $u=u^*\chi_{\Omega\times(0,T')}+\tilde u\chi_{\Omega\times[T',T)}$ are weak solutions to problem (\ref{ib-P}) fulfilling the required properties.

\end{proof}


\section{Rank-one smooth approximation under linear constraint}\label{sec:rank-1}
In this section, we prepare the main tool, Theorem \ref{thm:rank-1}, for proving the density result, Theorem \ref{thm:density}. Instead of presenting a special case that would be enough for our purpose, we exhibit the following result in a generalized and refined form of \cite[Lemma 2.1]{Po} that may be of independent interest (cf. \cite[Lemma 6.2]{MSv1}).

\begin{thm}\label{thm:rank-1}
Let $m,n\ge 2$ be integers, and let $A,B\in\M^{m\times n}$ be such that $\rank(A-B)=1$; hence
\[
A-B=a\otimes b=(a_i b_j)
\]
for some non-zero vectors $a\in\R^m$ and $b\in\R^n$ with $|b|=1.$
Let $L\in\M^{m\times n}$ satisfy
\begin{equation}\label{rank-1-1}
Lb\ne 0 \;\;\mbox{in}\;\;\R^m,
\end{equation}
and let $\mathcal{L}:\M^{m\times n}\to \R$ be the linear map defined by
\[
\mathcal{L}(\xi)=\sum_{1\le i\le m,\, 1\le j \le n} L_{ij}\xi_{ij}\quad \forall \xi\in\M^{m\times n}.
\]
Assume $\mathcal{L}(A)=\mathcal{L}(B)$ and $0<\lambda<1$ is any fixed number. Then there exists a linear partial differential operator $\Phi:C^1(\R^n;\R^m)\to C^0(\R^n;\R^m)$ satisfying the following properties:

(1) For any open set $\Omega\subset\R^n$,
\[
\Phi v\in C^{k-1}(\Omega;\R^m)\;\;\mbox{whenever}\;\; k\in\N\;\;\mbox{and}\;\;v\in C^{k}(\Omega;\R^m)
\]
and
\[
\mathcal{L}(\nabla\Phi v)=0 \;\;\mbox{in}\;\;\Omega\;\;\forall v\in C^2(\Omega;\R^m).
\]

(2) Let $\Omega\subset\R^n$ be any bounded domain. For each $\tau>0$, there exist a function $g=g_\tau\in  C^{\infty}_{c}(\Omega;\R^m)$ and two disjoint open sets $\Omega_A,\Omega_B\subset\subset\Omega$ such that
\begin{itemize}
\item[(a)] $\Phi g\in C^\infty_c(\Omega;\R^m)$,
\item[(b)] $\dist(\nabla\Phi g,[-\lambda(A-B),(1-\lambda)(A-B)])<\tau$ in $\Omega$,
\item[(c)] $\nabla \Phi g(x)= \left\{\begin{array}{ll}
                                (1-\lambda)(A-B) & \mbox{$\forall x\in\Omega_A$}, \\
                                -\lambda(A-B) & \mbox{$\forall x\in\Omega_B$},
                              \end{array}\right.$
\item[(d)] $||\Omega_A|-\lambda|\Omega||<\tau$, $||\Omega_B|-(1-\lambda)|\Omega||<\tau$,
\item[(e)] $\|\Phi g\|_{L^\infty(\Omega)}<\tau$,
\end{itemize}
where $[-\lambda(A-B),(1-\lambda)(A-B)]$ is the closed line segment in $\mathrm{ker}\mathcal{L}\subset\M^{m\times n}$ joining $-\lambda(A-B)$ and $(1-\lambda)(A-B)$.
\end{thm}

\begin{proof}
We mainly follow and modify the proof of \cite[Lemma 2.1]{Po} which is divided into three cases.

Set $r=\rank(L).$ By (\ref{rank-1-1}), we have $1\le r\le m\wedge n=:\min\{m,n\}.$

\textbf{(Case 1):} Assume that the matrix $L$ satisfies
\[
\begin{split}
L_{ij}=0\;\;& \mbox{for all $1\le i\le m,\, 1\le j\le n$ but possibly the pairs}\\
& \mbox{$(1,1),(1,2),\cdots,(1,n),(2,2),\cdots,(r,r)$ of $(i,j)$};
\end{split}
\]
hence $L$ is of the form
\begin{equation}\label{rank-1-5}
L=\begin{pmatrix} L_{11} & L_{12} & \cdots & L_{1r} &  \cdots & L_{1n}\\
                    & L_{22} & & & & & \\
                    & & \ddots & & & & \\
                    & & & L_{rr} & & & \\
                    & & & & & & \end{pmatrix}\in\M^{m\times n}
\end{equation}
and that
\[
A-B=a\otimes e_1\;\;\mbox{for some nonzero vector $a=(a_1,\cdots,a_m)\in\R^m$},
\]
where each blank component in (\ref{rank-1-5}) is zero.
From (\ref{rank-1-1}) and $\rank(L)=r$, it follows that the product $L_{11}\cdots L_{rr}\ne 0$. Since $0=\mathcal L(A-B)=\mathcal L(a\otimes e_1)=L_{11}a_1$, we have $a_1=0$.

In this case, the linear map $\mathcal L:\M^{m\times n}\to\R$ is given by
\[
\mathcal L(\xi)=\sum_{j=1}^n L_{1j}\xi_{1j}+\sum_{i=2}^r L_{ii}\xi_{ii},\quad \xi\in\M^{m\times n}.
\]
We will find a linear differential operator $\Phi:C^1(\R^n;\R^m)\to C^0(\R^n;\R^m)$ such that
\begin{equation}\label{rank-1-6}
\mathcal L(\nabla\Phi v)\equiv 0 \quad\forall v\in C^2(\R^n;\R^m).
\end{equation}
So our candidate for such a $\Phi=(\Phi^1,\cdots,\Phi^m)$ is of the form
\begin{equation}\label{rank-1-7}
\Phi^i v=\sum_{1\le k\le m,\,1\le l\le n}a^i_{kl}v^k_{x_l},
\end{equation}
where $1\le i\le m$, $v\in C^1(\R^n;\R^m)$, and $a^i_{kl}$'s are real constants to be determined; then for $v\in C^2 (\R^n;\R^m)$, $1\le i\le m$, and $1\le j\le n$,
\[
\partial_{x_j}\Phi^i v =\sum_{1\le k\le m,\,1\le l\le n}a^i_{kl}v^k_{x_lx_j}.
\]
Rewriting (\ref{rank-1-6}) with this form of $\nabla\Phi v$ for $v\in C^2 (\R^n;\R^m)$, we have
\[
\begin{split}
0 & \equiv  \sum_{1\le k\le m,\,1\le j,l\le n} L_{1j}a^1_{kl}v^k_{x_lx_j} + \sum_{i=2}^r\sum_{1\le k\le m,\,1\le l\le n} L_{ii}a^i_{kl}v^k_{x_lx_i}\\
& = \sum_{k=1}^m \Big(L_{11}a^1_{k1}v^k_{x_1x_1}+\sum_{j=2}^r (L_{1j}a^1_{kj}+L_{jj}a^j_{kj})v^k_{x_jx_j}+\sum_{j=r+1}^n L_{1j}a^1_{kj}v^k_{x_jx_j} \\
& \quad +\sum_{l=2}^r (L_{11}a^1_{kl}+L_{1l}a^1_{k1}+L_{ll}a^l_{k1})v^k_{x_lx_1} +\sum_{l=r+1}^n (L_{11}a^1_{kl}+L_{1l}a^1_{k1})v^k_{x_lx_1} \\
& \quad +\sum_{2\le j<l\le r} (L_{1j}a^1_{kl}+L_{1l}a^1_{kj}+L_{jj}a^j_{kl}+L_{ll}a^l_{kj})v^k_{x_lx_j}\\
& \quad +\sum_{2\le j\le r,\,r+1\le l\le n} (L_{1j}a^1_{kl}+L_{1l}a^1_{kj}+L_{jj}a^j_{kl})v^k_{x_lx_j}
\end{split}
\]
\[
\begin{split}
& \quad+\sum_{r+1\le j<l\le n} (L_{1j}a^1_{kl}+L_{1l}a^1_{kj})v^k_{x_lx_j} \Big).
\end{split}
\]

Should (\ref{rank-1-6}) hold, it  is thus sufficient to solve the following algebraic system for each $k=1,\cdots,m$ (after adjusting the letters for some indices):
\begin{eqnarray}
\label{rr-1}& L_{11}a^1_{k1}=0, &\\
\label{rr-4}& L_{1j}a^1_{kj}+L_{jj}a^j_{kj}=0 & \;\,\forall j=2,\cdots, r,\\
\label{rr-3}& L_{11}a^1_{kj}+L_{1j}a^1_{k1}+L_{jj}a^j_{k1}=0 & \;\,\forall j=2,\cdots, r, \\
\label{rr-5}& L_{1l}a^1_{kj}+L_{1j}a^1_{kl}+L_{ll}a^l_{kj}+L_{jj}a^j_{kl}=0 & \begin{array}{l}
                                                                    \forall j=3,\cdots, r, \\
                                                                    \forall l=2,\cdots, j-1,
                                                                  \end{array} \\
\label{rr-6}& L_{1j}a^1_{kj}=0 & \;\,\forall j=r+1,\cdots, n,\\
\label{rr-7}& L_{11}a^1_{kj}+L_{1j}a^1_{k1}=0 & \;\,\forall j=r+1,\cdots, n, \\
\label{rr-8}& L_{1l}a^1_{kj}+L_{1j}a^1_{kl}+L_{ll}a^l_{kj}=0 & \begin{array}{l}
                                                                    \forall j=r+1,\cdots, n, \\
                                                                    \forall l=2,\cdots, r,
                                                                  \end{array} \\
\label{rr-2}& L_{1l}a^1_{kj}+L_{1j}a^1_{kl}=0 & \begin{array}{l}
                                                                    \forall j=r+2,\cdots, n, \\
                                                                    \forall l=r+1,\cdots,j-1.
                                                                  \end{array}
\end{eqnarray}
Although these systems have infinitely many solutions, we will solve those in a way for a later purpose that the matrix $(a^j_{k1})_{2\le j, k\le m}\in\M^{(m-1)\times(m-1)}$ fulfills
\begin{equation}\label{rank-1-8}
a^j_{21}=a_j \quad\forall j=2,\cdots, m,\quad\mbox{and}\quad a^j_{k1}=0\quad\mbox{otherwise}.
\end{equation}
Firstly, we let the coefficients $a^i_{kl}\;(1\le i,k\le m,\,1\le l\le n)$ that do not appear in systems (\ref{rr-1})--(\ref{rr-2}) $(k=1,\cdots, m)$ be zero with an exception that we set $a^j_{21}=a_j$ for $j=r+1,\cdots,m$ to reflect (\ref{rank-1-8}). Secondly, for $1\le k\le m,\,k\ne 2$, let us take the trivial (i.e., zero) solution of system (\ref{rr-1})--(\ref{rr-2}). Finally, we take $k=2$ and solve system (\ref{rr-1})--(\ref{rr-2})  as follows with (\ref{rank-1-8}) satisfied.
Since $L_{11}\ne 0$, we set $a^1_{21}=0$; then (\ref{rr-1}) is satisfied. So we set
\[
a^j_{21}=-\frac{L_{11}}{L_{jj}}a^1_{2j},\;\;a^1_{2j}=-\frac{L_{jj}}{L_{11}}a_j \quad \forall j=2,\cdots,r;
\]
then (\ref{rr-3}) and (\ref{rank-1-8}) hold. Next, set
\[
a^j_{2j}=-\frac{L_{1j}}{L_{jj}}a^1_{2j}=\frac{L_{1j}}{L_{11}}a_j \quad\forall j=2,\cdots, r;
\]
then (\ref{rr-4}) is fulfilled. Set
\[
a^l_{2j}=-\frac{L_{1l}a^1_{2j}+L_{1j}a^1_{2l}}{L_{ll}}=\frac{L_{1l}L_{jj}a_j+L_{1j}L_{ll}a_l}{L_{ll}L_{11}},\;\;a^j_{2l}=0
\]
for $j=3,\cdots,r$ and $l=2,\cdots, j-1$; then (\ref{rr-5}) holds. Set
\[
a^1_{2j}=0\quad\forall j=r+1,\cdots,n;
\]
then (\ref{rr-6}) and (\ref{rr-7}) are satisfied. Lastly, set
\[
a^1_{2j}=0,\;\; a^l_{2j}=-\frac{L_{1j}}{L_{ll}}a^1_{2l}=\frac{L_{1j}}{L_{11}}a_l\quad\forall j=r+1,\cdots, n,\,\forall l=2,\cdots, r;
\]
then (\ref{rr-8}) and (\ref{rr-2}) hold. In summary, we have determined the coefficients $a^i_{kl}\;(1\le i,k\le m,\,1\le l\le n)$ in such a way that   system (\ref{rr-1})--(\ref{rr-2}) holds for each $k=1,\cdots, m$ and that (\ref{rank-1-8}) is also satisfied. Therefore, (1) follows from (\ref{rank-1-6}) and (\ref{rank-1-7}).

To prove (2), without loss of generality, we can assume $\Omega=(0,1)^n\subset\R^n.$ Let $\tau>0$ be given. Let $u=(u^1,\cdots,u^m)\in C^\infty(\Omega;\R^m)$ be a function to be determined. Suppose $u$ depends only on the first variable $x_1\in(0,1).$ We wish to have
\[
\nabla\Phi u(x)\in\{-\lambda a\otimes e_1,(1-\lambda) a\otimes e_1\}
\]
for all $x\in\Omega$ except in a set of small measure. Since $u(x)=u(x_1)$, it follows from (\ref{rank-1-7}) that for $1\le i\le m$ and $1\le j\le n$,
\[
\Phi^i u=\sum_{k=1}^m a^i_{k1} u^k_{x_1};\;\;\mbox{thus}\;\;\partial_{x_j}\Phi^i u=\sum_{k=1}^m a^i_{k1} u^k_{x_1 x_j}.
\]
As $a^1_{k1}=0$ for $k=1,\cdots, m$, we have $\partial_{x_j}\Phi^1 u =\sum_{k=1}^m a^1_{k1} u^k_{x_1 x_j}=0$ for $j=1,\cdots,n$. We first set $u^1\equiv 0$ in $\Omega$. Then from (\ref{rank-1-8}), it follows that for $i=2,\cdots, m$,
\[
\partial_{x_j}\Phi^i u =\sum_{k=2}^m a^i_{k1} u^k_{x_1 x_j}=a^i_{21}  u^2_{x_1 x_j}=a_i u^2_{x_1 x_j} = \left\{\begin{array}{ll}
                                  a_i u^2_{x_1 x_1} & \mbox{if $j=1$,} \\
                                  0 & \mbox{if $j=2,\cdots, n$.}
                                \end{array} \right.
\]
As $a_1=0$, we thus have that for $x\in\Omega$,
\[
\nabla\Phi u(x)=(u^2)''(x_1) a\otimes e_1.
\]
For irrelevant components of $u$, we simply take $u^3=\cdots =u^m\equiv 0$ in $\Omega$. Lastly, for a number $\delta>0$ to be chosen later, we choose a function $u^2(x_1)\in C^\infty_c(0,1)$ such that there exist two disjoint open sets $I_1,I_2\subset\subset (0,1)$ satisfying $||I_1|-\lambda|<\tau/2$, $||I_2|-(1-\lambda)|<\tau/2$, $\|u^2\|_{L^\infty(0,1)}<\delta$, $\|(u^2)'\|_{L^\infty(0,1)}<\delta$, $-\lambda\le (u^2)''(x_1)\le 1-\lambda$ for $x_1\in(0,1)$, and
\[
(u^2)''(x_1)= \left\{\begin{array}{ll}
                       1-\lambda & \mbox{if $x_1\in I_1$}, \\
                       -\lambda & \mbox{if $x_1\in I_2$}.
                     \end{array}
 \right.
\]
In particular,
\begin{equation}\label{rank-1-9}
\nabla \Phi u(x)\in[-\lambda a\otimes e_1,(1-\lambda)a\otimes e_1]\;\;\forall x\in\Omega.
\end{equation}
We now choose an open set $\Omega'_\tau\subset\subset\Omega':=(0,1)^{n-1}$ with $|\Omega'\setminus\Omega'_\tau|<\tau/2$ and a function $\eta\in C^\infty_c(\Omega')$ so that
\[
0\le\eta\le 1\;\;\mbox{in}\;\;\Omega',\;\; \eta\equiv 1\;\;\mbox{in}\;\Omega'_\tau,\;\;\mbox{and}\;\;|\nabla^i_{x'}\eta|<\frac{C}{\tau^i}\;\;(i=1,2)\;\;\mbox{in}\;\Omega',
\]
where $x'=(x_2,\cdots,x_n)\in\Omega'$ and the constant $C>0$ is independent of $\tau$.
Now, we define $g(x)=\eta(x') u(x_1)\in C^\infty_c(\Omega;\R^m)$. Set $\Omega_A=I_1\times\Omega'_\tau$ and $\Omega_B=I_2\times\Omega'_\tau.$ Clearly, (a) follows from (1). As $g(x)=u(x_1)=u(x)$ for $x\in \Omega_A\cup\Omega_B$, we have
\[
\nabla\Phi g(x)=\left\{\begin{array}{ll}
                         (1-\lambda)a\otimes e_1 & \mbox{if $x\in \Omega_A$}, \\
                       -\lambda a\otimes e_1 & \mbox{if $x\in \Omega_B$};
                       \end{array}
 \right.
\]
hence (c) holds. Also,
\[
||\Omega_A|-\lambda|\Omega||=||\Omega_A|-\lambda|=||I_1||\Omega'_\tau|-\lambda|=||I_1|-|I_1||\Omega'\setminus\Omega'_\tau|-\lambda|<\tau,
\]
and likewise
\[
||\Omega_B|-(1-\lambda)|\Omega||<\tau;
\]
so (d) is satisfied.
Note that for $i=1,\cdots,m,$
\[
\begin{split}
\Phi^i g & = \Phi^i(\eta u) = \sum_{1\le k\le m,\,1\le l\le n}a^i_{kl}(\eta u^k)_{x_l}=\eta\Phi^i  u+\sum_{1\le k\le m,\,1\le l\le n}a^i_{kl}\eta_{x_l} u^k\\
& = \eta\Phi^i  u+ u^2\sum_{l=1}^n a^i_{2l}\eta_{x_l} =\eta a^i_{21}u^2_{x_1} + u^2\sum_{l=1}^n a^i_{2l}\eta_{x_l}.
\end{split}
\]
So
\[
\|\Phi g\|_{L^\infty(\Omega)}\le C\max\{\delta,\delta\tau^{-1}\}<\tau
\]
if $\delta>0$ is chosen small enough; so (e) holds.
Next, for $i=1,\cdots,m$ and $j=1,\cdots,n,$
\[
\partial_{x_j}\Phi^i g=\eta_{x_j}a^i_{21}u^2_{x_1}+\eta\partial_{x_j}\Phi^i u + u^2_{x_j}\sum_{l=1}^n a^i_{2l}\eta_{x_l} + u^2\sum_{l=1}^n a^i_{2l}\eta_{x_l x_j};
\]
hence from (\ref{rank-1-9}),
\[
\dist(\nabla\Phi g,[-\lambda a\otimes e_1,(1-\lambda) a\otimes e_1])\le C\max\{\delta \tau^{-1},\delta\tau^{-2}\}<\tau\;\;\mbox{in $\Omega$}
\]
if $\delta$ is sufficiently small. Thus (b) is fulfilled.

\textbf{(Case 2):} Assume that $L_{i1}=0$ for all $i=2,\cdots, m$, that is,
\begin{equation}\label{rank-1-3}
L=\begin{pmatrix} L_{11} & L_{12} & \cdots & L_{1n}\\
                  0  & L_{22} & \cdots & L_{2n}\\
                  \vdots  & \vdots & \ddots & \vdots\\
                  0  & L_{m2} & \cdots & L_{mn} \end{pmatrix}\in\M^{m\times n}
\end{equation}
and that
\[
A-B=a\otimes e_1\;\;\mbox{for some nonzero vector $a\in\R^m$};
\]
then by (\ref{rank-1-1}), we have $L_{11}\ne 0.$

Set
\[
\hat L=\begin{pmatrix} L_{22} & \cdots & L_{2n}\\
                  \vdots & \ddots & \vdots\\
                  L_{m2} & \cdots & L_{mn} \end{pmatrix}\in\M^{(m-1)\times (n-1)}.
\]
As $L_{11}\ne 0$ and $\rank(L)=r$, we must have $\rank(\hat L)=r-1.$ Using the singular value decomposition theorem, there exist two matrices $\hat U\in O(m-1)$ and  $\hat V\in O(n-1)$ such that
\begin{equation}\label{rank-1-4}
\hat U^T\hat L\hat V=\mathrm{diag}(\sigma_2,\cdots,\sigma_r,0,\cdots,0)\in\M^{(m-1)\times (n-1)},
\end{equation}
where $\sigma_2,\cdots,\sigma_r$ are the positive singular values of $\hat L.$ Define
\begin{equation}\label{rank-1-2}
U=\begin{pmatrix} 1 & 0\\
                  0 & \hat U\end{pmatrix}\in O(m),\;\;
V=\begin{pmatrix} 1 & 0\\
                  0 & \hat V\end{pmatrix}\in O(n).
\end{equation}
Let $L'=U^T LV$, $A'=U^T AV$, and $B'=U^T BV$. Let $\mathcal{L}':\M^{m\times n}\to \R$ be the linear map given by
\[
\mathcal{L}'(\xi')=\sum_{1\le i\le m,\,1\le j\le n}L'_{ij}\xi'_{ij}\quad \forall \xi'\in\M^{m\times n}.
\]
Then, from (\ref{rank-1-3}), (\ref{rank-1-4}) and (\ref{rank-1-2}), it is straightforward to check the following:
\[
\left\{
\begin{array}{l}
  \mbox{$A'-B'=a'\otimes e_1$ for some nonzero vector $a'\in\R^m$,} \\
  \mbox{$L' e_1\neq 0$, $\mathcal L'(A)=\mathcal L'(B)$, and} \\
  \mbox{$L'$ is of the form (\ref{rank-1-5}) in Case 1 with $\mathrm{rank}(L')=r$.}
\end{array}\right.
\]
Thus we can apply the result of Case 1 to find a linear operator $\Phi':C^1(\R^n;\R^m)\to C^0(\R^n;\R^m)$ satisfying the following:

(1') For any open set $\Omega'\subset\R^n$,
\[
\Phi' v'\in C^{k-1}(\Omega';\R^m)\;\;\mbox{whenever}\;\; k\in\N\;\;\mbox{and}\;\;v'\in C^{k}(\Omega';\R^m)
\]
and
\[
\mathcal{L'}(\nabla\Phi' v')=0 \;\;\mbox{in}\;\;\Omega'\;\;\mbox{for all}\;\;v'\in C^2(\Omega';\R^m).
\]

(2') Let $\Omega'\subset\R^n$ be any bounded domain. For each $\tau>0$, there exist a function $g'=g'_\tau\in  C^{\infty}_{c}(\Omega';\R^m)$ and two disjoint open sets $\Omega'_{A'},\Omega'_{B'}\subset\subset\Omega'$ such that
\begin{itemize}
\item[(a')] $\Phi' g'\in C^\infty_c(\Omega';\R^m)$,
\item[(b')] $\dist(\nabla\Phi' g',[-\lambda(A'-B'),(1-\lambda)(A'-B')])<\tau$ in $\Omega'$,
\item[(c')] $\nabla \Phi' g'(x)= \left\{\begin{array}{ll}
                                (1-\lambda)(A'-B') & \mbox{$\forall x\in\Omega'_{A'}$}, \\
                                -\lambda(A'-B') & \mbox{$\forall x\in\Omega'_{B'}$},
                              \end{array}\right.$
\item[(d')] $||\Omega'_{A'}|-\lambda|\Omega'||<\tau$, $||\Omega'_{B'}|-(1-\lambda)|\Omega'||<\tau$,
\item[(e')] $\|\Phi' g'\|_{L^\infty(\Omega')}<\tau$.
\end{itemize}

For $v\in C^1(\R^n;\R^m)$, let $v'\in C^1(\R^n;\R^m)$ be defined by $v'(y)=U^T v(Vy)$ for $y\in\R^n$. We define $\Phi v(x)=U\Phi' v'(V^T x)$ for $x\in\R^n$, so that $\Phi v\in C^0(\R^n;\R^m).$ Then it is straightforward to check  that properties (1') and (2') of  $\Phi'$ imply respective properties (1) and (2) of  the linear operator $\Phi:C^1(\R^n;\R^m)\to C^0(\R^n;\R^m)$.

\textbf{(Case 3):} Finally, we consider the general case that $A$, $B$ and $L$ are as in the statement of the theorem. As $|b|=1$, there exists an $R\in O(n)$ such that $R^T b=e_1\in\R^n$. Also there exists a symmetric (Householder) matrix $P\in O(m)$ such that the matrix $L':=PLR$ has the first column parallel to $e_1\in\R^m$. Let
\[
A'=PAR\;\;\mbox{and}\;\; B'=PBR.
\]
Then $A'-B'=a'\otimes e_1$, where $a'=Pa\ne 0$. Note also that $L'e_1=PLRR^tb=PLb\ne 0$. Define $\mathcal L'(\xi')=\sum_{i,j}L'_{ij}\xi'_{ij}\;(\xi'\in\M^{m\times n})$; then $\mathcal L'(A')=\mathcal L(A)=\mathcal L(B)=\mathcal L'(B')$. Thus by the result of  Case 2, there exists a linear operator $\Phi':C^1(\R^n;\R^m)\to C^0(\R^n;\R^m)$ satisfying (1') and (2') above.

For $v\in C^1(\R^n;\R^m)$, let $v'\in C^1(\R^n;\R^m)$ be defined by $v'(y)=Pv(Ry)$ for $y\in\R^n$, and define $\Phi v(x)=P\Phi'v'(R^Tx)\in C^0(\R^n;\R^m)$. Then it is easy to check that the linear operator $\Phi:C^1(\R^n;\R^m)\to C^0(\R^n;\R^m)$ satisfies (1) and (2) by (1') and (2') similarly as  in  Case 2.

\end{proof}

\section{Proof of density result}\label{sec:density-proof}

In this final section, we prove Theorem \ref{thm:density}, which plays a pivotal role in the proof of the main results, Theorems \ref{thm:main-NCE} and \ref{thm:main-HEP}.

To start the proof, fix a $\delta>0$ and choose any $w=(u,v)\in\mathcal A$ so that $w\in W^{1,\infty}_{w^*}(\Omega_T;\R^2)\cap C^2(\bar\Omega_T;\R^2)$ satisfies the following:
\begin{equation}\label{func-w}
\left\{\begin{array}{l}
         \mbox{$w\equiv w^*$ in $\Omega_T\setminus\bar\Omega_T^w$ for some open set $\Omega_T^w\subset\subset\Omega_T^2$ with $|\partial\Omega_T^w|=0$,} \\
         \mbox{$\nabla w(x,t)\in U$ for all $(x,t)\in\Omega_T^2$, and $\|u_t-h\|_{L^\infty(\Omega_{T_\epsilon})}<\epsilon'$.}
       \end{array}
 \right.
\end{equation}
Let $\eta>0$. Our goal is to construct a function $w_\eta=(u_\eta,v_\eta)\in \mathcal A_\delta$ with $\|w-w_\eta\|_{L^\infty(\Omega_T)}<\eta;$ that is, a function $w_\eta\in W^{1,\infty}_{w^*}(\Omega_T;\R^2)\cap C^2(\bar\Omega_T;\R^2)$ with the following properties:
\begin{equation}\label{func-weta}
\left\{\begin{array}{l}
         \mbox{$w_\eta\equiv w^*$ in $\Omega_T\setminus\bar\Omega_T^{w_\eta}$ for some open set $\Omega_T^{w_\eta}\subset\subset\Omega_T^2$ with $|\partial\Omega_T^{w_\eta}|=0$,} \\
         \mbox{$\nabla w_\eta(x,t)\in U$ for all $(x,t)\in\Omega_T^2$,  $\|(u_\eta)_t-h\|_{L^\infty(\Omega_{T_\epsilon})}<\epsilon'$,} \\
         \mbox{$\int_{\Omega_T^2}\dist(\nabla w_\eta(x,t),K)\,dxdt\le\delta|\Omega_T^2|$, and $\|w-w_\eta\|_{L^\infty(\Omega_T)}<\eta$.}
       \end{array}
 \right.
\end{equation}

For clarity, we divide the proof into several steps.

\textbf{(Step 1):} Choose a nonempty open set $G_1\subset\subset\Omega_T^2\setminus\partial\Omega_T^w$ with $|\partial G_1|=0$ so that
\begin{equation}\label{int-1}
\int_{(\Omega_T^2\setminus\partial\Omega_T^w)\setminus G_1}\dist(\nabla w(x,t),K)\,dxdt\le\frac{\delta}{5}|\Omega_T^2|.
\end{equation}
Since $\nabla w\in U$ on $\bar G_1$, we have $\|u_t\|_{L^\infty(G_1)}<\gamma$;
then fix a number $\theta$ with
\begin{equation}\label{theta}
0<\theta<\min\{\epsilon'-\|u_t-h\|_{L^\infty(\Omega_{T_\epsilon})},\gamma-\|u_t\|_{L^\infty(G_1)}\}.
\end{equation}
For each $\mu>0$, let
\[
\begin{split}
G_2^\mu & =\{(x,t)\in G_1\,|\,\dist((u_x(x,t),v_t(x,t)),\partial\tilde U)>\mu\},\\
H_2^\mu & =\{(x,t)\in G_1\,|\,\dist((u_x(x,t),v_t(x,t)),\partial\tilde U)<\mu\},\\
F_2^\mu & =\{(x,t)\in G_1\,|\,\dist((u_x(x,t),v_t(x,t)),\partial\tilde U)=\mu\}.
\end{split}
\]
Since $\lim_{\mu\to 0^+}|H_2^\mu|=0,$ we can find a $\nu\in(0,\min\{\frac{\delta}{5},\theta\})$ such that
\begin{equation}\label{int-2}
\int_{H_2^{\nu}}\dist(\nabla w(x,t),K)\,dxdt\le\frac{\delta}{5}|\Omega_T^2|,\;\;G^\nu_2\not=\emptyset,\;\;\mbox{and}\;\; |F_2^{\nu}|=0.
\end{equation}
Let us write $G_2=G_2^{\nu}$ and $H_2=H_2^{\nu}$. Choose finitely many disjoint open squares $B_1,\cdots, B_N\subset G_2$ parallel to the axes such that
\begin{equation}\label{int-3}
\int_{G_2\setminus(\cup_{i=1}^N B_i)}\dist(\nabla w(x,t),K)\,dxdt\le\frac{\delta}{5}|\Omega_T^2|.
\end{equation}

\textbf{(Step 2):} Dividing the squares $B_1,\cdots, B_N$ into smaller disjoint sub-squares if necessary, we can assume that
\begin{equation}\label{small}
|\nabla w(x,t)-\nabla w(\tilde x,\tilde t)|<\frac{\nu}{8}
\end{equation}
whenever $(x,t),(\tilde x,\tilde t)\in B_i$ and $i=1,\cdots,N$. Now, fix any $i\in\{1,\cdots,N\}$. Let $(x_i,t_i)$ denote the center of the square $B_i$, and write
\[
(s_i,r_i)=(u_x(x_i,t_i),v_t(x_i,t_i))\in \tilde U;
\]
then $\dist((s_i,r_i),\partial\tilde U)>\nu$. Let $\alpha_i>0,\,\beta_i>0$ be  chosen so that
\[
(s_i-\alpha_i,r_i),(s_i+\beta_i,r_i)\in\tilde U,\;\;\dist((s_i-\alpha_i,r_i),\tilde K_-)=\frac{\nu}{2},\;\;\mbox{and}
\]
\[
\dist((s_i+\beta_i,r_i),\tilde K_+)=\frac{\nu}{2}.
\]
To apply Theorem \ref{thm:rank-1} with $m=n=2$ to the square $B_i$, let
\[
A_i=\begin{pmatrix} s_i-\alpha_i & b_i \\ b_i & r_i \end{pmatrix}\;\;\mbox{and}\;\;B_i=\begin{pmatrix} s_i+\beta_i & b_i \\ b_i & r_i\end{pmatrix},
\]
where $b_i=u_t(x_i,t_i);$ then
\[
A_i-B_i=\begin{pmatrix} -\alpha_i-\beta_i & 0 \\ 0 & 0 \end{pmatrix}=\begin{pmatrix} -\alpha_i-\beta_i  \\  0 \end{pmatrix}\otimes \begin{pmatrix} 1 \\  0 \end{pmatrix}.
\]
Let $\mathcal L:\M^{2\times 2}\to\R$ be the linear map defined by
\[
\mathcal L(\xi)=-\xi_{21}+\xi_{12}\quad\forall \xi\in\M^{2\times 2},
\]
with its associated matrix $L=\begin{pmatrix} 0 & -1 \\ 1 & 0 \end{pmatrix}$; then
\[
\mathcal L(A_i)=\mathcal L(B_i)(=0)\;\;\mbox{and}\;\;C_i=\lambda_i A_i+(1-\lambda_i) B_i,
\]
where $C_i=\nabla w(x_i,t_i)$ and $\lambda_i=\frac{\beta_i}{\alpha_i+\beta_i}\in(0,1).$ By Theorem \ref{thm:rank-1}, there exists a linear operator $\Phi_i:C^1(\R^2;\R^2)\to C^0(\R^2;\R^2)$ satisfying properties (1) and (2) in the statement of the theorem with $A=A_i$, $B=B_i$ and $\lambda=\lambda_i$. In particular, for the square $B_i\subset\R^2$ and a number $0<\tau<\min\{\frac{\nu}{8},\theta,\eta,\frac{\delta|\Omega_T^2|}{5SN}\}$ with $\begin{displaystyle} S:=\max_{r_1\le r\le r_2} (s_+(r)-s_-(r))>0 \end{displaystyle}$, we can find a function $g_i\in C^\infty_c(B_i;\R^2)$ such that
\begin{equation}\label{patch}
\left\{\begin{array}{l}
         \Phi_i g_i \in C^\infty_c(B_i;\R^2),\;\mathcal L(\nabla\Phi_i g_i)=0\;\;\mbox{in $B_i$,} \\
         \dist(\nabla\Phi_i g_i,[-\lambda_i(A_i-B_i),(1-\lambda_i)(A_i-B_i)])<\tau\;\;\mbox{in $B_i$,} \\
         |B_i^1|<\tau, \;\|\Phi_i g_i\|_{L^\infty(B_i)}<\tau ,
       \end{array}
 \right.
\end{equation}
where
\[
B_i^1=\{(x,t)\in B_i\,|\,\dist(\nabla\Phi_i g_i(x,t),\{-\lambda_i(A_i-B_i),(1-\lambda_i)(A_i-B_i)\})>0\}.
\]
Let $B_i^2=B_i\setminus B_i^1$.
We finally define
\[
w_\eta=w+\sum_{i=1}^N \chi_{B_i} \Phi_i g_i\;\;\mbox{in $\Omega_T$.}
\]

\textbf{(Step 3):} Let us check that $w_\eta=(u_\eta, v_\eta)$ is indeed a desired function satisfying (\ref{func-weta}). It is clear from (\ref{func-w}) and the construction above that $w_\eta\in W^{1,\infty}_{w^*}(\Omega_T;\R^2)\cap C^2(\bar\Omega_T;\R^2)$. Set $\Omega_T^{w_\eta}=\Omega_T^{w}\cup(\cup_{i=1}^N B_i)$; then $\Omega_T^{w_\eta}\subset\subset \Omega_T^2$, $|\partial\Omega_T^{w_\eta}|=0$, and $w_\eta=w=w^*$ in $\Omega_T\setminus\bar\Omega_T^{w_\eta}$. From (\ref{small}), (\ref{patch}), $\nu<\theta$ and $\tau<\nu/8$, it follows that for $i=1,\cdots, N$,
\[
\nabla w_\eta=\nabla w+\nabla\Phi_i g_i\in [A_i,B_i]_{\nu/4}\subset U\;\;\mbox{in $B_i$,}
\]
where $[A_i,B_i]_{\nu/4}$ is the $\frac{\nu}{4}$-neighborhood of the closed line segment $[A_i,B_i]$ in the space $\M^{2\times 2}_{sym}$; thus $\nabla w_\eta\in U$ in $\Omega_T^2$. By (\ref{theta})   and (\ref{patch}) with zero antidiagonal of $A_i-B_i$, we have
\[
\|(u_\eta)_t-h\|_{L^\infty(\Omega_{T_\epsilon})}\le
\|u_t-h\|_{L^\infty(\Omega_{T_\epsilon})}+\tau<\|u_t-h\|_{L^\infty(\Omega_{T_\epsilon})}+\theta<\epsilon',
\]
\[
\|w-w_\eta\|_{L^\infty(\Omega_{T})}=\|\sum_{i=1}^N \chi_{B_i} \Phi_i g_i\|_{L^\infty(\Omega_{T})}<\tau<\eta.
\]
Lastly, note that
\[
\begin{split}
\int_{\Omega_T^2} & \dist(\nabla w_\eta(x,t),K)\,dxdt = \int_{(\Omega_T^2\setminus\partial\Omega_T^w)\setminus G_1}\dist(\nabla w (x,t),K)\,dxdt\\
& + \int_{H_2}  \dist(\nabla w (x,t),K)\,dxdt + \int_{G_2\setminus(\cup_{i=1}^N B_i)}  \dist(\nabla w (x,t),K)\,dxdt\\
& +\sum_{i=1}^N \int_{B_i}  \dist(\nabla w (x,t)+\nabla \Phi_i g_i(x,t),K)\,dxdt=:I_1+I_2+I_3+I_4.
\end{split}
\]
Observe here that for $i=1,\cdots,N$,
\[
\begin{split}
\int_{B_i}  \dist( & \nabla w +\nabla \Phi_i g_i,K)\,dxdt  = \int_{B_i^1}  \dist(\nabla w +\nabla \Phi_i g_i,K)\,dxdt \\
& + \int_{B_i^2}  \dist(\nabla w (x,t)+\nabla \Phi_i g_i,K)\,dxdt \\
& \le S|B_i^1| + \nu|B_i^2| \le S\tau + \frac{\delta}{5}|B_i^2|\le \frac{\delta|\Omega_T^2|}{5N} + \frac{\delta}{5}|B_i^2|.
\end{split}
\]
Thus $I_4\le \frac{2\delta}{5}|\Omega_T^2|$; whence with (\ref{int-1}), (\ref{int-2}) and (\ref{int-3}), we have $I_1+I_2+I_3+I_4\le (\frac{\delta}{5}+\frac{\delta}{5}+\frac{\delta}{5}+\frac{2\delta}{5})|\Omega_T^2|=\delta|\Omega_T^2|$.

Therefore, (\ref{func-weta}) is proved, and the proof is complete.




\end{document}